\documentclass[12pt]{amsart}
\usepackage{verbatim,amscd,amssymb}
\usepackage[active]{srcltx}
\usepackage{graphicx}

\newtheorem{thm}{Theorem}[section]
\newtheorem{prop}[thm]{Proposition}
\newtheorem{lem}[thm]{Lemma}

\newtheorem{cor}[thm]{Corollary}

\newtheorem*{thmS}{Straightening Theorem}

\theoremstyle{definition}
\newtheorem{dfn}[thm]{Definition}

\def\0{\emptyset}
  \def\Cc{\mathcal{C}}
 
\def\Fc{\mathcal{F}}   \def\Mc{\mathcal{M}}
\def\Pc{\mathcal{P}}   
\def\Uc{\mathcal{U}} \def\Vc{\mathcal{V}}

\def\Uf{\mathfrak{U}} \def\Vf{\mathfrak{V}} 
  \def\Nf{\mathfrak{N}}

\def\Z{\mathbb{Z}}
\def\D{\mathbb{D}}
\def\C{\mathbb{C}}

\def\R{\mathbb{R}}

\renewcommand\emptyset{\varnothing}
\newcommand{\sm}{\setminus}

\def\eps{\varepsilon}
\def\ol{\overline}

\def\si{\sigma}  \def\ta{\theta}  
\def\al{\alpha}  \def\be{\beta}  \def\la{\lambda} \def\ga{\gamma}
\def\om{\omega}

\def\vp{\varphi}
\def\le{\leqslant}
\def\ge{\geqslant}
\def\wt{\widetilde}
\def\wh{\widehat}

\def\arg{\mathrm{arg}}

\def\en{\mathrm{end}}
\def\dia{\mathrm{diam}}

\def\imr{\mathrm{ImR}}
\def\inn{\mathrm{Int}}

\def\cuc{\mathcal{CU}}
\def\ch{\mathrm{CH}}
\def\uc{\mathbb{S}^1}
\def\bd{\mathrm{Bd}}

\def\<{\langle}
\def\>{\rangle}
\def\phd{\mathrm{PHD}}
\def\cdisk{\overline{\mathbb{D}}}
\def\thu{\mathrm{Th}}

\def\cont{\mathrm{Shad}}

\begin{document}
\date{\today}

\title[Immediate renormalization]
{Immediate renormalization of cubic complex
polynomials with empty rational lamination}

\dedicatory{Dedicated to Yulij Sergeevich Ilyashenko's 80th birthday}

\author[A.~Blokh]{Alexander~Blokh}

\author[L.~Oversteegen]{Lex Oversteegen}

\thanks{The second named author was partially  supported
by NSF grant DMS-1807558}

\author[V.~Timorin]{Vladlen~Timorin}

\thanks{The study has been funded within the framework of the HSE University Basic Research Program.}

\address[Alexander~Blokh and Lex~Oversteegen]
{Department of Mathematics\\ University of Alabama at Birmingham\\
Birmingham, AL 35294-1170}

\address[Vladlen~Timorin]
{Faculty of Mathematics\\
HSE University\\
6 Usacheva St., 119048 Moscow, Russia
}

\address[Vladlen~Timorin]
{Independent University of Moscow\\
Bolshoy Vlasyevskiy Pereulok 11, 119002 Moscow, Russia}

\email[Alexander~Blokh]{ablokh@math.uab.edu}
\email[Lex~Oversteegen]{overstee@math.uab.edu}
\email[Vladlen~Timorin]{vtimorin@hse.ru}



\begin{abstract}
A cubic polynomial $P$ with a non-repelling fixed point $b$
is said to be \emph{immediately renormalizable} if there exists
a (connected) QL invariant filled Julia set $K^*$ such that $b\in K^*$.
In that case, exactly one critical point of $P$ does not
belong to $K^*$. We show that if, in addition, the Julia set of $P$ has no
(pre)periodic cutpoints, then this critical point is recurrent.

\end{abstract}

\maketitle

\section{Introduction}

In the introduction, we assume knowledge of basics of complex dynamics.
Let $P$ be a monic non-linear polynomial with connected Julia sets $J(P)$.
An (external) ray of $P$ with a rational argument always lands at a point that is eventually
mapped to a repelling or parabolic periodic point. If two
external rays like that land at a point $x\in J(P)$, then
such rays are said to form a \emph{rational cut (at $x$)}. The family
of all rational cuts of a polynomial $P$ may be empty (then one says
that the \emph{rational lamination} of $P$ is empty); if it is
non-empty it provides a combinatorial tool allowing one to study
properties of $P$ even in the presence of Cremer or Siegel periodic points.

Consider quadratic polynomials with connected Julia set. It is known
 that any quadratic polynomial not in the closure of the quadratic \emph{Principal Hyperbolic Domain}
 (represented in the $c$-plane of polynomials $z^2+c$ by the interior and the boundary
 of the filled \emph{Main Cardioid})
 has rational cuts, which allows for powerful \emph{Yoccoz puzzles} techniques.
The purpose of this paper is to investigate a similar phenomenon in the
cubic case continuing a series of interconnected articles \cite{bopt14, bopt14b, bopt16a, bot16, bot22b}
in which parameter spaces of cubic polynomials and related topics are studied.

Consider a cubic polynomial $f$ without rational cuts.
Conjecturally, such $f$ must belong to the closure of the \emph{Principal Hyperbolic Domain}.
If not, then a number of pathological properties of $f$ are known:
in particular, by \cite[Theorem 1.10]{bot22b} the Julia set of $f$ is \emph{not} locally connected,
has positive measure and carries an invariant line field.
The main result of this paper is that if such a map $f$ is immediately renormalizable
(this concept is defined in the abstract and discussed thoroughly below)
and $\omega_2$ is the critical point of $f$ that does not participate in the renormalization,
 then $\omega_2$ is recurrent.

Our research is motivated as follows.
Firstly, it is a step towards proving the above conjecture:
we discover specific properties of the polynomials without rational cuts hoping that eventually this will lead to a contradiction
 (apart from Theorem \ref{t:recur} stated below,
 these properties include Theorem \ref{t:who-land} and some lemmas given in the text).
Secondly,
we develop tools for studying polynomials without rational cuts. In our view, this is interesting
because there are few approaches to studying such polynomials. Indeed, there are no Yoccoz puzzles to study the dynamics of $f$
(in particular, one cannot use a \emph{``divide and conquer''}
approach to recurrence of individual critical points). Moreover, Kiwi's rational lamination is empty and hence laminational tools
are not available either. Thus, what we consider is perhaps one of the simplest special cases when standard methods are not applicable.
We employ usual notation; e.g., $\bd(X)$ stand for the boundary of a subset $X\subset\C$, etc.

\begin{dfn}[\cite{DH-pl}]\label{d:ql1}
A \emph{polynomial-like (PL)} map is a proper ho\-lomorphic map $f: U\to f(U)$
of degree $k>1$, where $U\Subset f(U)$ are open Jordan disks.
The \emph{filled Julia set} $K(f)$ of $f$ is the
set of points in $U$ that never leave $U$ under iteration of $f$.
Let $\bd(K(f))=J(f)$ be the \emph{Julia set} of $f$.
Call $U$ a \emph{PL neighborhood} of $K(f)$ and assume that if $f$ is given,
then its basic neighborhood is fixed.
If $k=2$, then the corresponding maps (neighborhoods) are said to be
\emph{quadratic-like (QL)} maps (neighborhoods).
\end{dfn}

We can now state our main result.

\begin{thm} \label{t:recur}
Let $f$ be a cubic polynomi\-al with empty rational lamination that has
a QL restriction with a connected QL filled Julia set $K^*(f)=K^*$. Then the critical point of
$f$ that does not belong to $K^*$ is recurrent.
\end{thm}

In the situation of Theorem \ref{t:recur} we will always denote a connected QL filled Julia set
by $K^*$; also, we will fix its neighborhood $U^*$ on which $f$ is QL and denote $f|_{U^*}$ by $f^*$.

\noindent\textbf{Acknowledgements.} We are grateful to the referee for useful comments.

\section{Preliminaries}
\label{s:prelim}
By \emph{classes} of polynomials, we mean affine conjugacy classes. For
a polynomial $f$, let $[f]$ be its class, let $K(f)$ be its filled Julia set,
and let $J(f)$ be its Julia set. The \emph{connectedness locus $\Mc_d$ of degree $d$} is the
set of classes of degree $d$ polynomials whose critical points \emph{do
not escape} (i.e., have bounded orbits). Equivalently, $\Mc_d$ is the
set of classes of degree $d$ polynomials $f$ whose Julia set $J(f)$ is connected.
The classical \emph{Mandelbrot set} is identified with $\Mc_2$.
We study the \emph{cubic connectedness locus} $\Mc_3$.
The \emph{principal hyperbolic domain} $\phd_3$ of $\Mc_3$ is defined as the set of classes
of hyperbolic cubic polynomials whose Julia sets are Jordan curves.
Equivalently, $[f]\in\phd_3$ if both critical points of $f$ are
in the immediate basin attraction of the same (super-)attracting fixed point.
A polynomial is \emph{hyperbolic} if the orbits of all critical points converge
to (super-)attracting cycles.

\subsection{PL maps}\label{ss:pl}
Below is a brief overview of some results and concepts from \cite{DH-pl}.

\begin{dfn}
\label{d:ql2}
Two PL maps $f:U\to f(U)$ and $g:V\to g(V)$ of degree $k$ are said to be
\emph{hybrid conjugate} if there is a quasi-conformal map $\vp$, \emph{hybrid conjugacy}, from a
neighborhood of $K(f)$ to a neighborhood of $K(g)$ conjugating $f$ to
$g$ in the sense that $g\circ\vp=\vp\circ f$ wherever both sides are
defined and such that $\ol\partial\vp=0$ almost everywhere on $K(f)$.
\end{dfn}

Note that hybrid equivalent PL maps are, in particular, topologically conjugate on their filled Julia set.
Any polynomial $P$ can also be viewed as a PL map if one restricts $P$ to
 a suitable PL neighborhood of $K(P)$.

\begin{thmS}[\cite{DH-pl}]
Let $f:U\to f(U)$ be a PL map. Then $f$ is hybrid
conjugate to a polynomial $P$ of the same degree. Moreover, if $K(f)$
is connected, then $P$ is unique up to $($a global$)$ conjugation by an
affine map.
\end{thmS}

External rays of $P$ can be (partially) transferred to a neighborhood of $K(f)$ by a hybrid conjugacy.

\begin{dfn}\label{d:plrays}
Let $f$ be a polynomial, and for some Jordan disk $U^*$, the map $f^*=f|_{U^*}$ be PL.
Consider a monic polynomial $P$ hybrid equivalent to $f^*$.
The set $K(f^*)=K^*$ of $f^*$ is called the \emph{PL invariant filled Julia set}
 (or simply a \emph{PL set}).
Fix a hybrid conjugacy between $f^*$ and $P$.
The curves in $U^*$ corresponding (through the hybrid conjugacy) to dynamic rays of $P$ are called
 \emph{PL rays} of $f^*$. If the degree of $f^*$ is two, then we will talk about \emph{QL rays}.
Denote PL rays $R^*(\be)$, where $\be$ is the argument of the
external ray of $g$ corresponding to $R^*(\be)$.
We will also call them \emph{$K^*$-rays} to distinguish them from rays external to $K(f)=K$
called \emph{$K$-rays} and denoted by $R(\al)$ where $\al\in\R/\Z$ is the argument of the ray.
\end{dfn}


The $K^*$-rays are defined in a bounded neighborhood of $K^*$ while $K$-rays are unbounded.
Let $\vp:U^*\to\C$ be a hybrid conjugacy between $f^*$ and $P|_{\vp(U^*)}$.
We assume that $\vp$ is defined on $U^*$;
 this can be arranged by replacing $U^*$ with a smaller QL neighborhood of $K^*$ if necessary.
Composing $\vp|_{U^*\sm K^*}$ with a B\"ottcher parametrization of $\C\sm K(P)$
 gives a topological conjugacy $\psi^*$ between $f^*$ and the map $z\mapsto z^{\deg P}$ on $U^*\sm K^*$.
Note that $\psi^*$ is uniquely determined by $\vp$ only if $P$ is quadratic;
 otherwise there is a freedom in choosing a B\"ottcher parameterization of $\C\sm K(P)$.
The map $\psi^*$ conjugates $f$ with $z\mapsto z^{\deg(P)}$ near $K^*$.
In this paper, $P$ will be quadratic, and hence $\psi^*$ will depend only on the choice of
 the hybrid conjugacy $\vp$.
Using the map $\psi^*$, assign \emph{arguments} to all $K^*$-rays.
These are called \emph{quadratic arguments}.

Evidently, if $f$ is a polynomial of degree $d$ and $T\subsetneqq J(f)$
 is a proper PL invariant Julia set, then the degree of
$f|_T$ is less than $d$. In particular, if $f$ is a \emph{cubic}
polynomial and $K^*\subsetneqq K(f)$ is a PL invariant filled
Julia set, then the PL map $f|_{K^*}$ is QL. The
following lemma is proven in \cite{bot16} (it is based upon Theorem
5.11 from McMullen's book \cite{mcm94}).

\begin{lem}[Lemma 6.1 \cite{bot16}]\label{l:7.2}
Let $f$ be a complex cubic polynomial with a non-repelling fixed point $a$.
Then the QL invariant filled Julia set $K^*$ with $a\in K^*$ (if any) is unique.
\end{lem}

\subsection{Polynomials with empty rational lamination}
\label{ss:emp-lam}

As was said in the Introduction, we want to study cubic polynomials
$f\in \Mc_3$ without rational cuts (equivalently, with empty rational lamination).

\begin{lem}\label{l:gray-1}
Suppose that a cubic polynomial $f$ has empty rational lamination. Then
$f$ has exactly one fixed non-repelling point; all other
periodic points of $f$ are repelling. Moreover, there are no periodic
repelling/parabolic points of $f$ that are cutpoints of an invariant continuum $Q\subset J(f)$.
\end{lem}

\begin{proof}
If all fixed points of $f$ are repelling, then one of them
 is the landing point of more than one ray
 (indeed, there are 3 fixed points and only 2 invariant rays), a contradiction.
Also, if $f$ has a fixed non-repelling point and a distinct periodic non-repelling point,
 then the rational lamination of $f$ is non-empty by \cite[Theorem 3.3]{GM93}, a contradiction.
Finally, if $x$ is a repelling/parabolic cutpoint
 of an invariant continuum $Q\subset J(f)$, then, by the Main Theorem of \cite{bot21}, the point
$x$ is a cutpoint of $K(f)$, which implies (e.g., by Theorem 6.6 of \cite{mcm94})
that the rational lamination of $J(f)$ is non-empty, again a contradiction.
\end{proof}

\begin{cor}\label{c:7.2}
A cubic polynomial $f$ with empty rational lamination contains, in its filled Julia set
 $K(f)$, at most one forward invariant set $K^*$ that is the
 Julia set of a quadratic-like restriction of $f$;
 this set must contain the unique non-repelling fixed point of $f$.
\end{cor}

\begin{proof}
By Lemma \ref{l:gray-1}, the map $f$ has a unique non-repelling fixed point,
say, $a$, and all other periodic points of $f$ are repelling. If $K^*$ does not
contain $a$, then all its periodic points are repelling, and,
 by Theorem 7.5.2 of \cite{bfmot12} (or by the Straightening Theorem and simple counting), the map $f^*$
 has a fixed cutpoint $b$, a contradiction with Lemma \ref{l:gray-1}.
Thus, $K^*$ contains $a$; by Lemma \ref{l:7.2} it is unique.
\end{proof}

\subsection{Full continua and their decorations}\label{s:fucode}

In this section we consider \emph{pairs} $X\subset Y$ of full continua in the plane (a compact set $X\subset\C$ is
 \emph{full} if $\C\sm X$ is connected). This is a natural situation occurring in
complex dynamics, both when studying polynomials and their parameter spaces.
Indeed, let a cubic polynomial $f$ have a connected filled Julia set $K(f)=K$.
In this case if $K^*$ is a PL set of $f$, then $K^*\subset K$ is a pair of full continua.
Another example is when one takes the filled Main Cardioid of the Mandlebrot
set $\Mc_2$. It is easy to give other dynamical or parametric examples.

Let $X\subset Y$ be two full planar continua. We would like to
represent $Y$ as the union of $X$ and \emph{decorations (of $Y$ rel. $X$)}.

\begin{dfn}
\label{d:dec} Components of $Y\sm X$ are called \emph{decorations (of
$Y$ relative to $X$)}, or just \emph{decorations} (if $X$ and $Y$ are fixed).
\end{dfn}

Decorations are connected but not closed; thus, decorations may behave
differently from what common intuition suggests. In Lemma \ref{l:triv}
we discuss topological properties of decorations. Given sets $A$ and
$B$, say that $A$ \emph{accumulates in} $B$ if $\ol{A}\sm A\subset B$.

\begin{lem}\label{l:triv}
Any decoration $D$ of $Y$ rel. $X$ accumulates in $X$.
The set $\ol{D}\sm D=\ol{D}\cap X$ is a continuum.
The sets $\ol{D}$ and $D\cup X=\ol{D}\cup X$ are full continua.
\end{lem}

\begin{proof}
Suppose, by way of contradiction, that there exists $x\in \ol{D}\sm
(D\cup X)$. Then we have $D\subset A=D\cup \{x\}\subset \ol{D}$ while
$A\cap X=\0$. Since $D$ is connected, and since $D\subset A\subset \ol{D}$,
then $A$ is connected too. Hence $D$ is not a
component of $Y\sm X$, a contradiction.

The continuum $\ol{D}$ is full.
Indeed, otherwise there is a bounded complementary domain $U$ of $\ol D$.
If $U\sm Y\ne\0$, then all components of this set are bounded complementary domains of $Y$,
 a contradiction with $Y$ being full.
If $U\sm Y=\0$, then $U\cup D$ is a connected subset of $Y$,
 hence $U\subset D$, a contradiction with $U\cap\ol D=\0$.

Now, by the first paragraph, $\ol{D}\sm D=\ol{D}\cap X$ is compact.
Suppose that $\ol{D}\cap X$ is disconnected. Then there exists a
bounded component $U$ of $\C\sm (X\cup D)$ that at least partially
accumulates in $X$ and partially in $D$. Since $Y$ is full, then
$U\subset Y$; hence $U$ is a subset of a decoration that accumulates
(partially) to points of $D$. By the first paragraph this implies that
$U\subset D$, a contradiction. Thus, $\ol{D}\cap X$ is connected; then
$\ol{D}\cap X$ is a full continuum as both $X$ and $\ol{D}$ are full.
\end{proof}

Given a full continuum $X\subset \C$,
 we will use the inverse Riemann map $\psi: \C\sm X\to \C\sm \cdisk$
with real derivative at infinity.
Loosely, one can say that under the map $\psi$ the
continuum $X$ is replaced by the closed unit disk $\cdisk$ while the
rest of the plane is conformally deformed. Thus, under $\psi$ the
decorations become subsets of $\C\sm \cdisk$.

\begin{cor}\label{l:triv2}
Let $D$ be a decoration of $Y$ rel. $X$. Then $\ol{\psi(D)}\sm \psi(D)$ is a
 continuum $I_D\subset \uc$ (a circle arc, possibly degenerate, or the circle).
\end{cor}

\begin{proof}
Follows from Lemma \ref{l:triv}.
\end{proof}

Observe that the set $I_D$ can, indeed, coincide with the entire unit circle
(e.g., $D$ can spiral onto $\cdisk$). The arcs $I_D\ne \uc$ are
also possible as $\psi(D)$ may approach an arc $I_D$ by imitating the
behavior of the function $\sin(1/x)$ as $x\to 0^+$. Moreover, two
distinct decorations $D$ and $T$ may well have equal arcs $I_D$ and
$I_T$, or it may be so that, say, $I_D\subsetneqq I_T$, or $I_D$ and
$I_T$ can have a non-trivial intersection not coinciding with either
arc (all these examples can be constructed by varying the behavior of
components similar to the behavior function $\sin(1/x)$ as $x\to 0^+$).
However, there are some cases in which one can guarantee that each
decoration has a degenerate arc $I_D$.

\medskip

\noindent\textbf{Ray Assumption on $X$ and $Y$.}
\emph{
 There is a dense set $\mathcal A\subset \uc$ and a family of curves in $\C\sm\ol\D$
 each of which accumulates on a point of $\mathcal A$ and disjoint from $\psi(Y)$.}

\medskip

If the Ray Assumption holds for $X$ and $Y$, and there is a
neighborhood $U$ of $Y$ and a homeomorphism $\varphi:U\to W\subset \C$, then
the Ray Assumption holds for $\varphi(X)\subset \varphi(Y)$ too.

\begin{lem}\label{l:triv3}
Suppose that the Ray Assumption holds for $X\subset Y$.
Then, for every decoration $D$, the arc $I_D$ is degenerate.
\end{lem}

\begin{proof}
Suppose that $I_D$ is a non-degenerate arc.
Choose two points $x$, $y$ so that $I_D$ intersects both components of $\uc\sm\{x,y\}$.
The set $\psi(D)$ is contained in one of the two disjoint open regions formed by the curves $R_x$ and $R_y$
 in the complement of the closed unit disk.
It follows that $\psi(D)$ can accumulate to only one of the circle arcs formed by the points $x$, $y$, a contradiction.
\end{proof}

Since we study decorations in the complex dynamical setting,
making the Ray Assumption is not overly restrictive.

\begin{lem}\label{l:raya1}
Suppose that $K^*$ is a connected invariant filled PL Julia set contained
in a connected filled Julia set $K$ of a polynomial $P$. Then
$K^*\subset K$ satisfy the Ray Assumption.
\end{lem}

\begin{proof}
Choose a periodic repelling point $x\in K^*$ and a $K$-ray $R$ landing at $x$.
Consider the map $\psi^*:U^*\sm K^*\to\C\sm\cdisk$ introduced above
 (it conjugates $f$ with $z\mapsto z^2$).
The curve $\psi^*(R)$ lands on a point of $\uc$, and these points are dense in $\uc$.
The remark after we define the Ray Assumption now shows, that
$K^*\subset K$ satisfy it.
\end{proof}

\section{Cubic parameter slices}
\label{s:para-sli}
Let $\Fc$ be the space of monic polynomials
$$
f_{\lambda,b}(z)=\lambda z+b z^2+z^3,\quad \lambda\in \C,\quad b\in \C.
$$
An affine change of variables reduces any cubic polynomial to the form
$f_{\lambda,b}$. Clearly, $0$ is a fixed point of every polynomial in
$\Fc$. Define the \emph{$\la$-slice} $\Fc_\lambda$ of $\Fc$ as the
space of all polynomials $g\in\Fc$ with $g'(0)=\lambda$, i.e.,
polynomials $f(z)=\lambda z+b z^2+z^3$ with fixed $\lambda\in \C$.
Also, denote by $\Fc_{nr}$
the space of polynomials $f_{\lambda,b}$ with $|\la|\le 1$ (``nr''
from ``\textbf{n}on-\textbf{r}epelling'').
For a fixed $\la$ with $|\la|\le 1$ the \emph{$\la$-connectedness locus
$\Cc_\lambda$}, of the \emph{$\la$-slice} of the cubic connectedness
locus is defined  as the set of all $f\in\Fc_\la$ such that $K(f)$ is connected.
This is a full continuum \cite{BrHu, Z}. We study sets
$\Cc_\la\subset \Fc_\la$ as we want to investigate to what extent
results concerning the quadratic Mandelbrot set $\Mc_2$ hold for $\Cc_\la$.

Fig. \ref{fig:slices} shows slices $\Fc_\la$ for some values of $\la$,
 together with the corresponding connectedness loci $\Cc_\la$. 

\begin{figure}
  \centering
  \includegraphics[width=\textwidth]{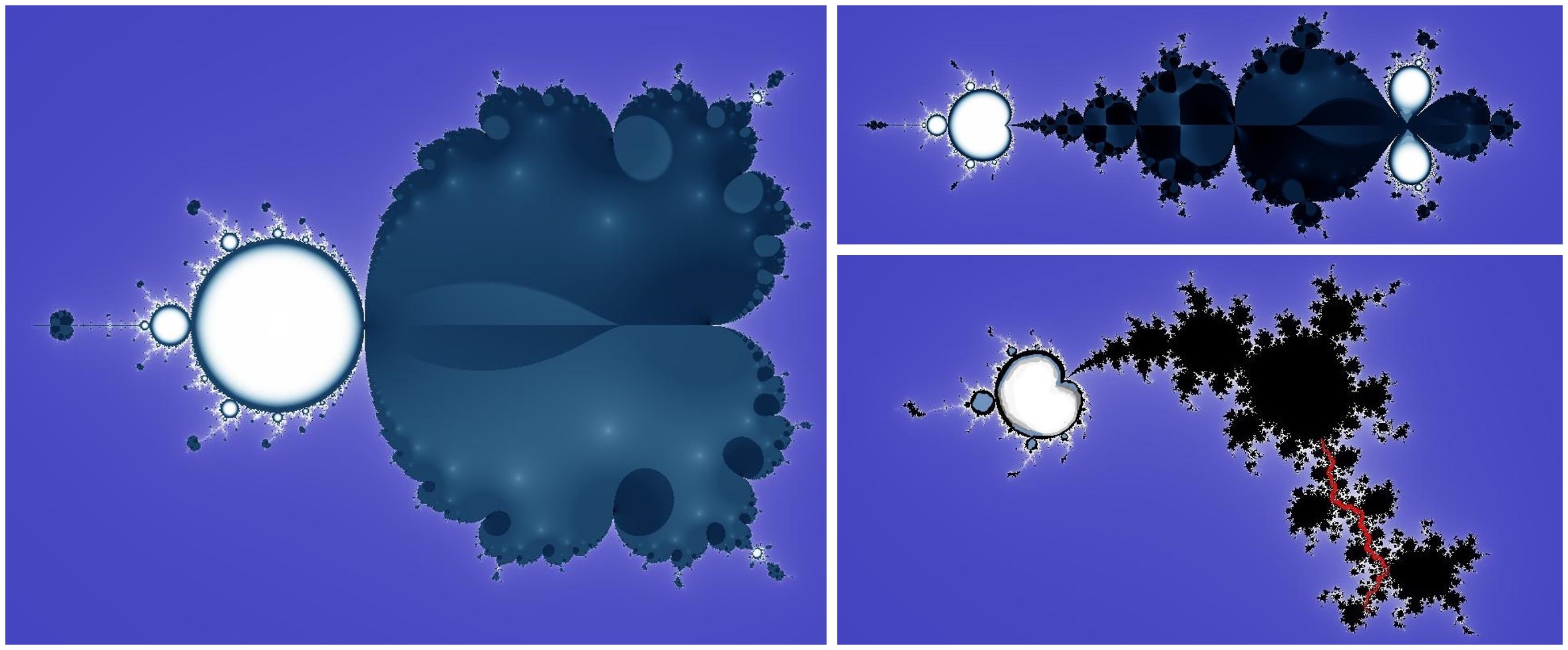}
  \caption{\small Some parameter slices $\Fc_\la$.
  The figures show the planes of parameter $a=b^2$, rather than $b$, to get rid of trivial symmetry.
  ``Background'' points represent the complement of $\Cc_\la$.
  Left: $\la=1$; the distorted ``cauliflower'' on the right is $\Pc_\la$; whereas the white ``bulb'' on the left
  is where a renormalization copy of the Mandelbrot set starts; the ``bulb'' itself representing the (baby) main cardioid.
  Top right: $\la=-1$; the distorted ``fat basilica'' represents $\Pc_\la$.
  Bottom right: $\la=e^{\pi i\sqrt{2}}$; the set $\Pc_\la$ looks like $K(\la z+z^2)$,
  in which the central Fatou component is replaced with a simple arc (\emph{Zakeri curve}, cf. \cite{Z} and \cite{bost22}).}
  \label{fig:slices}
\end{figure}

\subsection{Immediately renormalizable polynomials}
\label{ss:imre-vs-hyp}
Let us describe small perturbations of QL invariant filled Julia
sets $K^*\ni 0$ of $f\in \Fc_{nr}$ that contain $0$ (assuming $K^*$ exists for a given $f$).
The \emph{quadratic representative} of $f^*$ is
 the polynomial $z^2+c$ hybrid conjugate to $f^*$.

\begin{lem}\label{l:cnct} Let $f\in \Fc_{nr}$ be a polynomial, $K^*$
be a QL invariant filled Julia set containing $0$. Then
$K^*$ is connected. Every cubic polynomial $g\in \Fc_{nr}$ sufficiently
close to $f$ has a QL Julia set $K^*(g)$ containing $0$; the
set $K^*(g)$ is also connected. Moreover, if $0$ is an attracting
fixed point for $g$, then $g$ has a QL Julia set, which is a
Jordan curve; in particular, $[g]\notin \phd_3$.
\end{lem}

\begin{proof}
Since $0$ is non-repelling, then, by the Fatou-Shishikura inequality,
the critical point of the quadratic representative of $f^*$ cannot escape.
Hence $K^*$ is connected. Let $f^*:U^*\to V^*, f^*=f|_{U^*}$ be the
associated QL map.
Replacing $U^*$ with a smaller PL neighborhood if necessary, assume that
 there are no critical points of $f$ on the boundary of $V^*$.
If $g$ is very close to $f$, then $0\in U^*$ and, moreover,
 there is a new Jordan disk $W^*$ such that $g:W^*\to V^*$ is a 2-1 branched covering.
By the above, the associated QL Julia set $K^*(g)$ is connected.
Finally, if $f_i\to f$ are polynomials with $0$
as an  attracting fixed point, then, by the above, for large $i$, the
polynomial $f_i$ has a QL filled Julia set coinciding with
the closure of the basin of immediate attraction of $0$.
Therefore, $[f_i]\notin \phd_3$ for large $i$, as desired.
\end{proof}

Call a cubic polynomial $f\in\Fc_{nr}$ \emph{immediately renormalizable}
 if there are Jordan domains $U^*$ and $V^*$ such that
$0\in U^*$, and $f:U^*\to V^*$ is a QL map; denote by
$K^*$ the filled QL Julia set of $f^*=f|_{U^*}$ (in
 what follows we \emph{always} use the notation $U^*,$ $V^*,$ $f^*$ and $K^*$ when talking about
immediately renormalizable polynomials).
Denote the set of all \textbf{im}mediately
\textbf{r}enormalizable polynomials by $\imr$, and let $\imr_\la=\Fc_\la\cap
\imr$. Let $\Pc$ be the set of polynomials $f\in\Fc_{nr}$ such that there are polynomials $g\in\Fc$ arbitrarily
close to $f$ with $|g'(0)|<1$ and $[g]\in\phd_3$. Then clearly
$[f]\in\ol\phd_3$ (observe, that there may be polynomials outside of
$\Pc$ whose classes are also in $\ol\phd_3$). Also, set
$\Pc_\la=\Pc\cap \Fc_\la$.
In Fig. \ref{fig:slices}, subsects $\Pc_\la\subset\Cc_\la$ can be seen, as explained in the caption.

Corollary \ref{c:cnct} follows from Lemma \ref{l:cnct}.

\begin{cor}\label{c:cnct}
If $f\in \imr$, then $K^*$ is connected.
The set $\imr$ is open in $\Fc_{nr}$.
The set $\imr_\la$ is open in $\Fc_\la$ for any $\la$, $|\la|\le 1$.
The sets $\imr$ and $\Pc$ are disjoint.
\end{cor}

We want to study the sets $\imr$ and $\Pc$; Corollary \ref{c:cnct} shows
that they are disjoint, so this investigation may be done in parallel.
The sets $\imr_\la$ and $\Pc_\la$ are well understood
 in the case $|\la|<1$.
D. Faught in his 1992 thesis \cite{fau92} sketched a proof of the fact that $\Pc_0$ is a Jordan disk.
A formal and complete argument along the same lines was given later by P. Roesch \cite{roe06}.
Results of L. Tan and C. Petersen \cite{tp06} (based on a holomorphic motion
 between parameter slices) allow one to justify that the description of $\Pc_0$
 carries over to $\Pc_\la$ for all $\la$ in the open unit disk.

\begin{thm}[\cite{fau92,roe06,tp06}]\label{t:roesch}
The set $\Pc_\la$ is a Jordan disk for any $\la$ with $|\la|<1$.
\end{thm}

Let us combine this with \cite{bopt16a} where sufficient
conditions on polynomials for being immediately renormalizable are given.

\begin{thm}[\cite{bopt16a}]\label{t:when-imr}
If $f\in \Fc_\la$, $|\la|\le 1$, belongs to the unbounded complementary
 component of $\Pc_\la$ in $\Fc_\la$, then $f$ is immediately renormalizable.
\end{thm}

These theorems and Lemma \ref{l:cnct} imply Corollary \ref{c:<1}.

\begin{cor}\label{c:<1}
If $|\la|<1$, then $\imr_\la=\C\sm \Pc_\la$.
\end{cor}

Some polynomials are guaranteed to belong to $\phd_3$.

\begin{lem}[cf. \cite{bostv23}]\label{l:lambda<0}
If $|\la|<1$, then $[z^3+\la z]\in \phd_3$.
\end{lem}

\begin{proof}
Consider $f(z)=\la z+z^3$.
We claim that $J(f)$ is a Jordan curve. Let $U$ be the basin
of immediate attraction of $0$ (which is an attracting fixed point of $f$).
Since $f^n(-z)=-f^n(z)$ for every $n$, then $U$ is centrally
symmetric with respect to $0$. Since there exists a
critical point $c\in U$ and $f'(-z)=f'(z)$ for any $z$, then
$-c$ is critical. The central symmetry of $U$ with respect to $0$ now implies
that $-c\in U$. Since both critical points of $f$
belong to $U$, the claim follows.
\end{proof}

\begin{thm}\label{l:0in}
For any $\la$ with $|\la|<1$, we have $0\in \inn(\Pc_\la)$.
For any $\la$ with $|\la|\le 1$, we have $0\in \Pc_\la$, and $\Pc_\la$ is a continuum.
\end{thm}

\begin{proof}
The first claim is proven in Lemma \ref{l:lambda<0}.
To prove the rest, observe that if $|\la|=1$, then $\Pc_\la=\limsup
\Pc_\tau$ where $\tau\to \la, |\tau|<1$. Since $0\in
\inn(\Pc_\tau)$ for all such numbers $\tau$, we see that $0\in
\Pc_\la$, and $\Pc_\la$, being the lim sup
of the continua $\Pc_\tau$, which all share a common point $0$, is also
a continuum as claimed.
\end{proof}

\subsection{The structure of the slice $\Fc_\la$}
\label{ss:stru-fla}

Following \cite{bot16}, define the set $\cuc_\la, |\la|\le 1$ as the
set  of all polynomials $f\in \Fc_\la$ with connected Julia sets and
such that the following holds:

\begin{enumerate}

\item $f$ has no repelling periodic cutpoints in $J(f)$;

\item $f$ has at most one non-repelling cycle different from $0$, and,
 if such a cycle exists, its multiplier is $1$.

\end{enumerate}

The set $\cuc_\la$ is a centerpiece, literally and figuratively, of
the $\la$-slice $\Cc_\la$ of the cubic connectedness locus.
Conjecturally, $\cuc_\la=\Pc_\la$; in particular, there is no way of distinguishing these sets in Fig. \ref{fig:slices}.
A big role in studying polynomials from $\Cc_\la$ is played by studying properties
of the quadratic polynomial $z^2+\la z$ whose fixed point $0$
has multiplier $\la$.
We will assume that $|\la|=1$ but $\la$ is not a root of unity
 (no extra arithmetic conditions are imposed on $\la$).

For a closed subset $A\subset \uc$ of at least 3 points,
 call its convex hull $\ch(A)$ (taken with respect to the standard real affine structure of the plane) a \emph{gap}.
Given a chord $\ell=\ol{ab}$ of the unit circle with
endpoints $a$ and $b$, set
$\si_3(\ell)=\ol{\si_3(a)\si_3(b)}$ (we abuse the notation and identify the angle-tripling
map $\si_3:\R/\Z\to \R/\Z$ with the map $z^3:\uc\to \uc$; we treat the map $\si_2$ similarly). For a closed set $A\subset \uc$, call
each complementary arc of $A$ a \emph{hole} of $A$. Given a compactum $A\subset \C$ let the
\emph{topological hull} $\thu(A)$ be the complement to the unbounded complementary domain of $A$.

\subsection{Invariant quadratic gaps}\label{sss:family}
Let us discuss properties of \emph{quadratic} $\si_3$-invariant gaps \cite{BOPT}.
For our purposes it suffices to consider gaps $G$ such that $G\cap \uc$ has no isolated points.
\emph{Edges} of $G$ are chords on the boundary of $G$, and \emph{holes} of $G$
 are components of $\uc\sm G$.
The gap $G$ being ($\si_3$)-\emph{invariant}
 means that an edge of a gap $G$ maps to an edge of $G$, or to a point in $G\cap \uc$;
 \emph{quadratic} means that after collapsing holes of $G$
the map $\si_3|_{\bd(G)}$ induces a locally strictly monotone two-to-one map of the unit
circle to itself that preserves orientation and has no critical points.

For convenience, normalize the length of the circle so that it equals
$1$. Let $\Vf$ be a quadratic $\si_3$-invariant gap with no isolated
points. Then there is a unique open arc $I_{\Vf}$ of $\uc$ (called the \emph{major hole} of $\Vf$)
 complementary to $\Vf\cap \uc$ whose length is greater
than or equal to $1/3$; the length of this arc is at most $1/2$.
The edge $M_{\Vf}$ of $\Vf$ connecting the endpoints of $I_{\Vf}$ is called the \emph{major} of $\Vf$.
If $M_{\Vf}$ is \emph{critical} (that is, the endpoints of $M_{\Vf}$ have the same $\si_3$-image),
 then $M_{\Vf}$ and $\Vf$ are said to be of \emph{regular critical type};
 if $M_{\Vf}$ is periodic under $\si_3$, then $M_{\Vf}$ and $\Vf$ are said to be of \emph{periodic type}.
It follows from \cite[Lemma 3.10]{BOPT} that $M_{\Vf}$ is of one of these two types.
Collapsing edges of $\Vf$ to points, we construct
a monotone map $\tau:\Vf\to \uc$ that semiconjugates
$\si_3|_{\bd(\Vf)}$ and $\si_2:\uc\to \uc$.

The map $\tau$ is uniquely defined by the fact that it is monotone and
semiconjugates $\si_3|_{\bd(\Vf)}$ and $\si_2:\uc\to \uc$. Indeed, if
there had been another map $\tau'$ like that, then there would have
existed a non-trivial orientation preserving homeomorphism of the
circle to itself conjugating $\si_2$ with itself.
However, it is easy to see that the only such map is the identity
(recall that $\si_2$ is an expanding covering that has a unique fixed point).

\begin{thm}[\cite{bot16, bot22b}]\label{t:cubio}
If $|\la|\le 1$, then $\thu(\Pc_\la)=\cuc_\la$.
\end{thm}

\section{Decorations and their quadratic arguments}
The following are standard notions of the Carath\'eodori theory,
 applicable to any full continuum $K^*\subset\C$.
A \emph{crosscut} (of $K^*$) is a closed arc $I$ with endpoints $x, y\in K^*$ such that $I\sm \{a, b\}\subset \C\sm K^*$.
If $a_n$ is a crosscut, then the \emph{shadow} $\cont(a_n)$ of the
crosscut $a_n$ is the bounded complementary component of $a_n\cup K^*$.
A sequence of crosscuts $a_n, n=1, 2, \dots$ is \emph{fundamental} if
$a_{n+1}\subset \cont(a_n)$ for every $n$, and the diameter of $a_n$
converges to $0$ as $n\to\infty$. Two fundamental sequences of crosscuts
are \emph{equivalent} if crosscuts of one sequence are eventually
contained in the shadows of crosscuts of the other one, and vice versa.
This is an equivalence relation on the set of fundamental sequences of crosscuts whose classes are called \emph{prime ends}
 of $\C\sm K^*$.
In what follows the set of endpoints of a closed arc $I$ is denoted by $\en(I)$.

\subsection{Quadratic arguments}
\label{ss:quadarg}
Let $f$ be an immediately renormalizable cubic polynomial with filled QL Julia set $K^*$.
 Fix a choice of a QL neighborhood of $K^*$, which defines $K^*$-rays.
Consider $K^*$-rays $R^*(\al)$. Clearly, $f(R^*(\al))\supset R^*(2\al)$
(the curve $f(R^*(\al))$ extends the ray $R^*(2\al)$ into the annulus
between the QL neighborhood of $K^*$ and its image).

By the Carath\'eodory theory, every $K^*$-ray $R^*(\al)$
corresponds to a certain prime end $E^*(\al)$ of $\C\sm K^*$ represented by a fundamental sequence
of crosscuts $\{a_n\}$. For every $a_n$, a \emph{tail} of $R^*(\al)$ is contained in $\cont(a_n)$ (a
\emph{tail} of $R^*(\al)$ is defined by a point $x\in R^*(\al)$ and
is the component of $R^*(\al)\sm \{x\}$ that accumulates
in $K^*$). The (useful) associated picture in
$\C\sm \cdisk$ is obtained by transferring the picture from the $K^*$-plane to
$\C\sm \cdisk$ by means of the map $\psi^*$ introduced right after Definition \ref{d:plrays}.

Namely, for every sufficiently large $n$ (so that $a_n\subset U^*$), the set $\psi^*(a_n\sm \en(a_n))$ is an arc
$I_n\subset \C\sm \cdisk$ without endpoints such that $\ol{I_n}$ is a closed arc
with endpoints $x_n, y_n\in \uc$. One can choose
the circle arc $I'_n$ positively oriented from $x_n$ to $y_n$ such that
$z_\al\in I'_n$ where $z_\al\in \uc$ is the point of the circle with argument
$\al$. Consider the Jordan curve $Q_n=I_n\cup
I'_n$; then the radial ray $R_\al$ with  initial point
at $z_\al\in \uc$, intersected with the simply connected domain $U_{a_n}$ with boundary
$Q_n$, contains a small subsegment of $R_\al$ with  endpoint $z_\al$.
Observe that $(\psi^*)^{-1}(U_{a_n})$ is the shadow of the crosscut $a_n$.
The \emph{impression} of $E^*(\al)$ is the intersection of the closures of $\cont(a_n)$. We
say that a prime end $E^*(\al)$
 is \emph{disjoint} from a set $S\subset\C\sm
K^*$ if $\cont(a_n)\cap S=\0$ for all sufficiently large $n$.
In what follows, when talking about crosscuts, we will use this notation.

Recall that $\psi^*$ is defined only on $U^*\sm K^*$.
If $X\subset\C\sm K^*$ is any subset, then $\psi^*(X)$ means $\psi^*(X\cap U^*)$.

\begin{lem}\label{l:cross-primends}
Suppose that $X\subset \C\sm K^*$ is a connected set, and
$\psi^*(X)$ accumulates on exactly one point $z_\al\in \uc$ with
argument $\al$. Then $X$ is non-disjoint from $E^*(\al)$ and
disjoint from any other prime end.
\end{lem}

\begin{proof}
Let $a$ be a crosscut associated with $E^*(\al)$. Consider the set $U_a$. Since
$\psi^*(X)$ accumulates on $z_\al$, then $\psi^*(X)$ is non-disjoint
from $U$, and hence $X$ is non-disjoint from $\cont(a)$. By definition, $X$ is
non-disjoint from  $E^*(\al)$.
Also, for any point $t=e^{2\pi \beta i}\ne z_\al$, we can find a crosscut
$b$ associated with $\beta$ and so small that $\ol{U_b}$ is
disjoint from $\psi^*(X)$. Then $X$ is disjoint from
$\cont(b)$ and hence $X$ is disjoint from $E^*(\beta)$.
\end{proof}

Recall that a set $A$ \emph{accumulates in} $B$ if $\ol{A}\sm A\subset B$.

\begin{prop} \label{p:pends}
Every decoration $D$ is disjoint from all prime ends of $K^*$ except
exactly one.
\end{prop}

\begin{proof}
Consider the connected set $\psi^*(D)\subset \C\sm \cdisk$. It is disjoint from $\psi^*$-images of $K$-rays
that land on (pre)periodic points of $K^*$. Hence there is a dense in $\uc$ set $A\subset \uc$ such
that, for each $x\in A$, there exists the $\psi^*$-image $\Gamma$ of a $K$-ray that lands on $x$.
Clearly, $\Gamma$ lies in a curve that extends all the way to infinity while being disjoint from $\uc$
 and from the $\psi^*$-images of all decorations.
It now follows easily that
$\psi^*(D)$ accumulates on exactly one point of $\uc$ and the lemma follows.
\end{proof}

We are ready to define \emph{quadratic arguments} of decorations.

\begin{dfn}\label{d:qarg}
Let $D$ be a decoration and $E^*(\al)$ is the only prime end non-disjoint from $D$
(the above is justified by Proposition \ref{p:pends}).
Then $\al$ is called the \emph{quadratic argument} of $D$ and is
denoted by $\arg_2(D)$.
\end{dfn}

By Proposition \ref{p:pends}, quadratic arguments of decorations are well
defined while different decorations may a priori have the same
quadratic arguments.
Using the map $\psi^*:U^*\sm K^*\to \C\sm\cdisk$ we can transfer
all decorations to the set $\C\sm \cdisk$; then, for any decoration $D$,
the set $\psi^*(D)$ accumulates on the point $e^{2\pi i \arg_2(D)}$ of the
unit circle with argument $\arg_2(D)$ (that is, $\psi^*(D)$
accumulates to only one point of the unit circle with argument $\arg_2(D)$).

\subsection{Dynamics on decorations}
\label{s:deco}
We consider $f\in \imr$ (then, by Lemma \ref{l:7.2}, the PL set $K^*$ is unique)
such that the critical point $\om_2$ of $f$ that does not belong to $K^*$ is not recurrent.
Also, a \emph{pullback} under a continuous map $g:\C\to\C$
of a connected set $A\subset\C$ is a component of $g^{-1}(A)$,
and an \emph{$n$-th pullback} (or just \emph{iterated pullback}) of $A$ is a pullback of
$A$ under $g^n$.

\begin{lem}
\label{l:pull}
There exists a pullback $\wt K^*$ of $K^*$ disjoint from
 $K^*$ and such that $f:\wt K^*\to K^*$ is a homeomorphism.
\end{lem}

\begin{proof}
Let $U^*$ be a PL neighborhood of $K^*$, then $U^*$ is a component of $f^{-1}(V^*)$, where $V^*=f(U^*)$.
The map $f:f^{-1}(V^*)\to V^*$ is proper, hence a branched covering of degree 3.
Since the degree of $f:U^*\to V^*$ is two, the remaining component $\wt U^*$ of $f^{-1}(V^*)$
 maps by $f$ homeomorphically onto $V^*$.
The lemma follows.
\end{proof}

The notation $\wt K^*$ will be used from now on. Also, from now on
by \emph{decorations} we mean those of $K$ rel. $K^*$.

\begin{dfn}
A decoration is said to be \emph{critical} if it contains $\wt K^*$.
Thus there is only one critical decoration denoted $D_c$.
All other decorations are said to be \emph{non-critical}.
\end{dfn}

Let $v_2=f(\om_2)$ be the critical value associated with the point $\om_2$.

\begin{lem}
\label{l:v-nin-k}
Neither $\om_2$ nor $v_2$ belong to $K^*$.
\end{lem}

\begin{proof}
Clearly, $\om_2\notin K^*$. If $v_2\in
K^*$, then there are two preimages of $v_2$ in $K^*$ and two preimages
of $v_2$ outside of $K^*$ (both numbers take multiplicities into
account). This contradicts $f$ being three-to-one.
\end{proof}

To study iterated pullbacks of $K^*$ we use an important result of R. Ma\~n\'e.

\begin{thm}[\cite{man93}]
\label{t:mane} If $f:\ol\C\to\ol\C$ is a rational map and $z\in\ol\C$ a
point that does not belong to the limit set of any recurrent critical
point, then, for some $C>0$, some $0<q<1$, and some Jordan disk $W$ around $z$,
the spherical diameter of any component of $f^{-n}(W)$ is less than $Cq^n$.
\end{thm}

Neighborhoods from  Theorem \ref{t:mane} are called \emph{Ma\~n\'e neighborhoods}.

\begin{thm}\label{t:kcollapse}
The iterated pullbacks of $K^*$ form a null sequence.
\end{thm}

\begin{proof}
By Theorem \ref{t:mane}, by Lemma \ref{l:pull}, and by our assumptions there exist $0<q<1$, $C>0$, and a finite cover of $\wt K^*$
by open Jordan disks $W_1, \dots, W_N$ such that, for any disk $U$ contained in some disk $W_i$, we have that any $n$-th pullback of $U$
is of spherical diameter less than $Cq^n$. Choose $N$ so that $2q^N<1/2$. Then choose a tight neighborhood $V=V_0\subset \bigcup W_i$
of $\wt K^*$ so that $f^N(V)\cap V=\0$.
We may assume that $V=\bigcup_{j=1}^{k_0} U_j$ where $U_j$ are open Jordan disks that are
Ma\~n\'e neighborhoods. Consider pullbacks $V_1, V_2, \dots$ of $V$.
For each $n$ the set $V_n$ is the union of, say, $k_n$
pullbacks of the sets $U_j$ each of which is of diameter at most $Cq^n$. Hence the diameter of $V_n$ is at most $Ck_nq^n$.
The number $k_n$ stays the same as long as the pullbacks of $V$ do not contain $\om_2$, and almost doubles on the step when the
previous pullback of $V$ contains $f(\om_2)$ and, therefore, the next pullback of $V$ contains $\om_2$.
Denote by $n_i$ the $i$-th moment with $\om_2\in V_{n_i}$. This means that $f^{n_i}(\om_2)\in V$ and
$f^{n_{i+1}}(\om_2)\in V$, which implies that $n_{i+1}-n_i\ge N$.

Let us now estimate the diameter of $V_m$ from above. We have $\dia(V_0)\le Ck_0$. We claim that
$$\dia(V_m)\le Ck_0\cdot (1/\sqrt[N]{2})^m.$$
Indeed, one can argue by induction. Then on each step when the pullback of $V$ does not contain
$\om_2$ the immediately preceding upper bound is (by Theorem \ref{t:mane}) multiplied
by $q<1/\sqrt[N]{2}$. Now, consider $\dia(V_{n_{i+1}})$. By induction
$$\dia(V_{n_i})\le Ck_0\cdot (1/\sqrt[N]{2})^{n_i},$$
and since $2q^N<1/2$, then we have
$$\dia(V_{n_{i+1}})\le
Ck_0 (2q^N) \left(\frac 12\right)^{\frac{n_{i+1}-n_i}N-1}\left(\frac 12\right)^{\frac {n_i}N}\le
Ck_0\cdot (1/\sqrt[N]{2})^{n_{i+1}},
$$
as desired.
\end{proof}

\begin{lem}\label{l:deconto}
Suppose that $D$ is a decoration and $E$ is a pullback of $D$. Then $f(E)=D$.
\end{lem}

\begin{proof}
Observe that $f:\C\sm (K^*\cup\wt K^*)\to\C\sm K^*$ is a branched covering,
 and $E$ is a pullback of $D$ under this branched covering.
It follows that $E$ is mapped under $f$ onto a closed and open subset of $D$
 (since a branched covering is an open and closed map), i.e., onto $D$ itself.
\end{proof}

Let us now consider the map $\Psi$ that collapses $K^*$
 and all iterated pullbacks of $K^*$ to points; the notation introduced below
 will be used in what follows.
Since the grand orbit of $K^*$ is a null sequence of full continua,
by Moore's Theorem, $\Psi(\C)$ can be identified with $\C$, so that
$\Psi: \C\to \C$ is a monotone onto map. It is easy to see that it induces a branched covering map
$f_\Psi:\C\to \C$; moreover, $K^*$ then maps to a fixed critical repelling point $\Psi(K^*)=a^*$
 (the map $g$ is, of course, not smooth).
The set $\wt a^*=\Psi(\wt K^*)$ is also a point since $\wt K^*$ is a pullback of $K^*$.
Abusing notation, we will often identify points $x$ and singletons $\{x\}$ in what follows.

\begin{dfn}
  Let $\Sigma$ be a simple closed curve in $\wh\C:=\C\cup\{\infty\}$ containing points $a^*$ and $\infty$.
Suppose that $\Sigma$ is disjoint from $\Psi(K\sm K^*)$ but separates this set.
Then $\Sigma$ is called a \emph{$\Psi$-cut} (of $\Psi(K\sm K^*)$).
A \emph{$\Psi$-sector} is defined as any complementary component to a $\Psi$-cut; this is an open Jordan disk.
Define a \emph{preimage $\Psi$-sector} as the $\Psi$-preimage of a $\Psi$-sector.
Clearly, a preimage $\Psi$-sector is an open topological disk but not necessarily a Jordan domain.
\end{dfn}

Recall the following topological result playing an important role in polynomial dynamics.

\begin{thm}[Theorem 6.6 \cite{mcm94}]
\label{t:mcm}
For any pair of decorations $D\ne D'$, there is a $\Psi$-sector containing $\Psi(D)$ and disjoint from $\Psi(D')$.
Hence there is a preimage $\Psi$-sector containing $D$ and disjoint from $D'$.
\end{thm}

Let $f\in \imr$; then $K^*$ is unique by Lemma \ref{l:7.2}.
Assume that $V^*$ is very tight around $K^*$.
Let $\om_1$ be the critical point of $f$ belonging to $K^*$; let
$\om_2$ be the other critical point of $f$ (this notation will be used in what follows).
To indicate the dependence on $f$, we may write $K^*(f)$, $\om_2(f)$, etc.
We emphasize that in Section \ref{s:deco} we put \emph{no} restrictions on the rational lamination of $f$.
For $x\notin K^*$, let $D(x)$ be the decoration containing $x$; set
$D_v=D(v_2)$ and call it \emph{critical value decoration}.
Set $L$ to be $\{\om_2\}$ (if $\om_2\in J(f)$), or the closure of the Fatou domain of
$f$ containing $\om_2$ (if any).

Initial dynamical properties of decorations are listed in Theorem \ref{t:decor}.

\begin{thm}\label{t:decor}
The critical decoration $D_c$ maps onto $K(f)$ while any other decoration maps onto some decoration one-to-one.
Any decoration different from $D_v$ has three homeomorphic pullbacks;
 two of them are decorations, and one accumulates in $\wt K^*$.
The decoration $D_v$ has a homeomorphic pullback $D'_v$, which is itself a decoration, and a
 pullback that maps onto $D_v$ two-to-one, contains
 $\om_2$, is contained in $D_c$, and accumulates in both $K^*$ and $\wt K^*$.
\end{thm}

\begin{proof}
Evidently, decorations are $\Psi$-preimages of the components of the set $X=\Psi(K\sm K^*)$.
The proof is divided into steps.

\smallskip

\noindent \textbf{Step 1.} \emph{If $D$ is a decoration of $f$, then
every pullback of $D$ is a subset of some decoration of $f$. Moreover,
if $D'$ and $D''$ are decorations and $f(D')\cap D''\ne \0$, then
$f(D')\supset D''$}.

\smallskip

\noindent \emph{Proof of Step 1}.
Let $E$ be a pullback of $D$.
Clearly, $E\subset K\sm K^*$.
Since $E$ is connected, it must lie in some decoration.
Now, if $D'$ and $D''$ are decorations and $f(D')\cap D''\ne \0$,
 then we can choose a pullback $E''$ of $D''$, which is non-disjoint from $D'$.
By the above, $E''\subset D'$, and, by Lemma \ref{l:deconto}, we have $f(E'')=D''$.

\smallskip

\noindent \textbf{Step 2.} \emph{If $D$ is a non-critical decoration, then $f(D)$ is a decoration.}

\smallskip

\noindent \emph{Proof of Step 2.} The set $f(D)$ is connected and disjoint
from $K^*$ (by definition of a non-critical decoration). Hence it is
contained in one decoration. By Step 1, the set $f(D)$ coincides with this
decoration.

\smallskip

\noindent \textbf{Step 3.} \emph{Suppose that $D$ is not a critical value decoration.
Then all $f$-pullbacks of $D$ map forward one-to-one. Only two of these pullbacks are decorations;
the remaining pullback accumulates in $\wt K^*$.}

\smallskip

\noindent\emph{Proof of Step 3.}
Let $S$ be a preimage $\Psi$-sector containing $D$ and disjoint from $v_2=f(\omega_2)$
(it exists by Theorem \ref{t:mcm}).
Then $S$ contains no critical values of $f$, and there are three homeomorphic pullbacks of $S$.
Two of these pullbacks accumulate in $K^*$ and one accumulates in $\wt K^*$.

\smallskip

\noindent\textbf{Step 4.} \emph{Suppose that $D=D_v$ is the critical value decoration.
Then one pullback of $D$ is a decoration mapping one-to-one onto $D$, and the other
 pullback connects $K^*$ with $\wt K^*$ and maps forward two-to-one.}

\smallskip

\noindent\emph{Proof of Step 4.}
Similarly to step 3, consider a $\Psi$-sector $S_\Psi$ bounded by a $\Psi$-cut $\Sigma$ of $\C$.
Then $S_\Psi$ contains the critical value $\Psi(v_2)$ of $f_\Psi$, hence there are two pullbacks of $S_\Psi$:
 one pullback is attached to $a^*$ and maps forward one-to-one;
 the other pullback is bounded by two pullbacks of $\Sigma$ passing through $a^*$ and $\wt a^*$.
The statement follows.
\end{proof}

Recall that $\psi^*$ conjugates $f$ with $z\mapsto z^2$ near $K^*$.

\begin{lem}\label{l:z2}
If $D\ne D_c$ is a decoration, then $\arg_2(f(D))=2\arg_2(D)$.
In the notation of Theorem $\ref{t:decor}$, we have that $\arg_2(D_v)=2\arg_2(D_c)$.
\end{lem}

Corollary \ref{c:2pbdec} follows from Proposition \ref{p:pends}.

\begin{cor} \label{c:2pbdec}
Let $D\not=D_v$ be a decoration with quadratic argument $\al$.
Then both elements of $\si_2^{-1}(\al)$ are the quadratic arguments of the decorations
containing pullbacks of $D$.
These decorations are distinct.
\end{cor}

\section{Combinatorics of renormalization}
\label{s:cubarg}
Consider a polynomial $f\in \imr\cap \Fc_\la$ with $|\la|\le 1$;
 as before denote by $f^*:U^*\to V^*$ the corresponding QL map.

\subsection{Invariant quadratic gap $\Uf$}
In \cite{bot16} we defined an invariant quadratic gap $\Uf(f)=\Uf$ associated with $f$ and $f^*$.
 When $J(f)$ is disconnected, gaps analogous to
$\Uf$ were studied in \cite{bclos} where tools from
\cite{lp96} were used.
Recall that by Lemma \ref{l:7.2} if $f$ is a cubic
polynomial with a non-repelling fixed point $a$, then there exists at
most one QL invariant filled Julia set $K^*$ containing
$a$; by Corollary \ref{c:7.2}, if $f$ has empty rational lamination, then
it has a unique non-repelling fixed point $a$ and at most one
QL invariant filled Julia set that, if it exists, must contain $a$.
Recall also that $\wt K^*$ is a component of $f^{-1}(K^*)$ different from $K^*$.

From now on we fix an immediately renormalizable
polynomial $f\in \Fc_{nr}$ and do not refer to $f$ in our notation
(we write $\Uf$ instead of $\Uf(f)$ etc).
Lemma \ref{l:inv-quad} summarizes some results of \cite[Section 7]{bot16}.
Properties of $\Uf$ listed in this lemma define $\Uf$ uniquely and
 can be taken as a definition.

\begin{lem}\label{l:inv-quad}
The gap $\Uf$ is an invariant quadratic gap of regular critical or periodic type.
If a (pre)periodic external ray $R(\alpha)$ of $f$ lands in $K^*$, then $\alpha\in\Uf$.
If a (pre)periodic external ray $R(\beta)$ of $f$ lands in $\wt K^*$, then $\beta$ is
 in the closure of the major hole of $\Uf$.
\end{lem}

Consider the map $\tau$ defined in Section \ref{ss:stru-fla}
 for any quadratic invariant gap of $\si_3$
 (this map collapses the edges of $\Uf$ and semiconjugates
$\si_3|_{\bd(\Uf)}$ with $\si_2$).
Observe that $K$-rays with arguments from $\bd(\Uf)$ do not necessarily
have principal sets contained in $K^*$. Nevertheless the map $\tau$
allows us to relate decorations of $K^*$ and their quadratic arguments
with edges and vertices of the gap $\Uf$.

\begin{lem}\label{l:cridi3}
The quadratic argument of $D_c$ is $\tau(M_{\Uf})$.
\end{lem}

\begin{proof}
If the quadratic argument of $D_c$ is not $\tau(M_{\Uf})$, then there is an edge/vertex $y$ of $\Uf$
such that $\tau(y)$ is the quadratic argument of $D_c$ and
we can find periodic angles $\al'$, $\be'\in \Uf$ such that the circle arc $I=(\al', \be')$
 contains $y$ but does not contain the endpoints of $M_{\Uf}$.
For $K$-rays $R(\al')$, $R(\be')$ with arguments $\al'$, $\be'$, consider the
component $W$ of $\C\sm [R(\al')\cup R(\be')\cup K^*]$ containing
$K$-rays with arguments from $I$.
By Lemma \ref{l:inv-quad}, all periodic $K$-rays landing in $\wt K^*$ lie in the closure of the major hole of $\Uf$.
Hence $\widetilde K^*$ is disjoint from $W$.
On the other hand, by definition of the quadratic argument,
 $D_c\supset\wt K^*$ must be contained in $W$, a contradiction.
\end{proof}

Similarly to decorations, one can define quadratic arguments of $K$-rays landing in $K^*$.
It follows, similarly to Lemma \ref{l:z2}, that
$$
\arg_2(R(3\ga))=2\arg_2(R(\ga))
$$
for any $\ga$.
Also, as follows from Lemma \ref{l:cridi3}, the quadratic arguments of $R(\al_\Uf)$ and $R(\be_\Uf)$ are both equal to $\tau(\al_\Uf)=\tau(\be_\Uf)$,
 where $(\al_\Uf,\be_\Uf)$ is the major hole of $\Uf$.
On the other hand, there is only one monotone map from $\uc$ to $\uc$ collapsing $M_{\Uf}$ to a point
 and semi-conjugating $\si_3$ with $\si_2$ on $\Uf$; this map is $\tau$.
It follows that $\arg_2(R(\ga))=\tau(\ga)$ for any $K$-ray $R(\ga)$ landing in $K^*$.

\begin{thm}\label{t:who-land}
If $\al\in \Uf$ is a (pre)periodic angle that never maps to an
endpoint of the major $M_{\Uf}$ of $\Uf$, then the $K$-ray $R(\al)$ with
argument $\al$ lands on a point of $K^*$.
\end{thm}

\begin{proof}
  Set $\ta=\tau(\al)$, and let $R^*(\ta)$ be the $K^*$-ray with argument $\ta$.
It lands on some point $x\in K^*$; we claim that $R(\al)$ lands on the same point.
Let $\be$ be the argument of a $K$-ray landing on $x$ such that the ray is homotopic to $R^*(\ta)$ rel. $K^*$
 (it follows from the Main Theorem of \cite{bot21} that such $\be$ exists).
By definition of the quadratic argument, $\arg_2(R(\be))=\ta$.
On the other hand, by the above, $\arg_2(R(\be))=\tau(\be)$; it follows that $\al$ and $\be$ have the same $\tau$-images.
In particular, $\al=\be\in \Uf$ unless both $\al$ and $\be$ are the endpoints of an edge of $\Uf$.
However, the latter is impossible by the assumption.
\end{proof}

\subsection{Sectors}
As before, suppose that $f\in\imr\cap\Fc_\la$ with $|\la|\le 1$ is fixed.
Consider a pair of external rays $R(\al)$, $R(\be)$ landing in $K^*$.
The set $\Sigma(\al,\be)=K^*\cup R(\al)\cup R(\be)$ divides the plane into two components, one
of which contains all external rays with arguments in $(\al,\be)$ and the other contains all external rays with arguments in $(\be,\al)$.
To formally justify this claim, collapse $K^*$ to a point (i.e., consider the equivalence relation $\sim$
on $\ol\C$, whose classes are $K^*$ and single points in $\ol\C\sm K^*$).
By Moore's theorem, the quotient space $\ol\C/\sim$ is homeomorphic to the sphere.
The image of $\Sigma(\al, \be)$ under the quotient projection, together with the image of the point at infinity, form a Jordan curve.
The statement now follows from the Jordan curve theorem.
Let $S^\circ(\al,\be)$ be the component of $\C\sm\Sigma(\al,\be)$ containing all external rays with arguments in $(\al,\be)$.
Observe that $S^\circ(\al,\be)$ is defined only if the rays $R(\al)$, $R(\be)$ both land in $K^*$.
The sets $S^\circ(\al,\be)$ will be called \emph{open sectors}, and the sets $\Sigma(\al,\be)$ will be called \emph{cuts}.
Images of sectors contain $K^*$ if and only if sectors contain $\wt K^*$.

An open sector $S^\circ(\al,\be)$ is associated with its \emph{argument arc} $(\al,\be)\subset\R/\Z$
that consists of arguments of all rays included in $S^\circ(\al,\be)$.
Note, that this sector does not have to coincide with the union of those rays as
 open sectors may contain decorations. More generally, consider a subset $T\subset\C$.
The set $T$ is said to be \emph{$(f)$-radial} if any ray intersecting $T$ lies in $T$.
For a radial set $T$ we can define the \emph{argument set} $\arg(T)$ of $T$ as the set of all $\ga\in\R/\Z$ with $R(\ga)\subset T$.
Every open sector is a radial set, whose argument set is an open arc.
It is clear that, for any radial set $T$, we have
$
\arg(f(T))=\si_3(\arg(T))$ and $\arg(f^{-1}(T))=\si_3^{-1}(\arg(T))$.

\begin{lem}
  \label{l:pb-opsec}
 Let $S^\circ$ be an open sector and let $T^\circ$ be an $f$-pullback of $S^\circ$.
Then $\arg(T^\circ)$ is the union of $m$
components of $\si_3^{-1}(\arg(S^\circ))$ for some $m\in\{1,2\}$.
The number of critical points in $T^\circ$ equals $m-1$.
If $\om_2\notin T^\circ$ and the closure of $T^\circ$ intersects $K^*$, then $T^\circ$ is an open sector mapping 1-1 onto $S^\circ$.
Any pullback of $S^\circ$ is disjoint from $K^*$.
\end{lem}

\begin{proof}
The first claim ($\arg(T^\circ)$ is a union of components of $\si_3^{-1}(\arg(S^\circ))$) is immediate.
Since $f:T^\circ\to S^\circ$ is proper, then this map has a well-defined degree
equal to the number of components in $\arg(T^\circ)$. By the Riemann--Hurwitz formula,
the degree equals the number of critical points in $T^\circ$ plus one.
Thus the first two claims of the lemma follow.

Let us prove the third claim. The only critical point that can lie in $T^\circ$ is $\om_2$.
Since we assume that $\om_2\notin T^\circ$, then $\arg(T^\circ)$ has only one component. Let $\arg(T^\circ)=(\al,\be)$.
Then $T^\circ$ is bounded by $R(\al)\cup R(\be)$ and a part of $K^*\cup \wt K^*$, and $\bd(T^\circ)$ is connected.
If both $R(\al)$ and $R(\be)$ land in $K^*$, then, by definition, $T^\circ$ coincides with the open sector $S^\circ(\al,\be)$.
If both $R(\al)$ and $R(\be)$ land in $\wt K^*$, then $T^\circ$ is disjoint from $K^*$
as its image $S^\circ$ does not contain $K^*$; this implies that $\ol{T^\circ}$ is disjoint with $K^*$, a contradiction
with our assumptions. Since $\bd(T^\circ)$ is connected, this exhausts all possibilities and completes the proof of the lemma.
\end{proof}

Lemma \ref{l:img-sec} deals with the images of sectors.

\begin{lem}
  \label{l:img-sec}
  Consider an open sector $S^\circ(\al,\be)$, whose argument arc is mapped one-to-one under $\si_3$.
Then $f(S^\circ(\al,\be))=S^\circ(3\al,3\be)$.
Moreover, $S^\circ(\al,\be)$ maps one-to-one onto $S^\circ(3\al,3\be)$.
\end{lem}

\begin{proof}
  Let $T^\circ$ be the $f$-pullback of $S^\circ(3\al,3\be)$ that includes rays with arguments in $(\al,\be)$.
Clearly, the rays $R(\al)$, $R(\be)$ are on the boundary of $T^\circ$.
Since these rays land in $K^*$ and $T^\circ\cap K^*=\0$, then $T^\circ\subset S^\circ(\al,\be)$.
Since by the assumptions of the lemma $\om_2\notin T^\circ$, then
$T^\circ=S^\circ(\al,\be)$ by Lemma \ref{l:pb-opsec}, as desired.
\end{proof}

\section{Backward stability}
In this section we study backward stability of decorations and show
that under certain circumstances decorations shrink as we pull them back.

\begin{lem}
  \label{l:sn}
  Fix $q\in (0,1)$ and $b>0$.
  Consider a sequence of positive numbers $s_n$ such that either $s_{n+1}=qs_n$ or $s_{n+1}\le 2qs_n+b$.
In the latter case call $n$ a \emph{bad} subscript.
Suppose that the distance between consecutive bad subscripts tends to infinity and denote bad subscripts by $n_i$.
Then $s_{n_i}\to 0$ as $i\to \infty$.
\end{lem}

\begin{proof}
It suffices to show that $s_n\to 0$ as $n$ runs through all bad indices $n_1<n_2<\dots$;
fix $\eps>0$ and $N$ such that $q^N<1/8$ and $q^{N-1}b<\eps$.
Then, for $i$ large, we have $n_{i+1}-n_i\ge N$, and
$$
s_{n_{i+1}}=q^{n_{i+1}-n_i-1}(2qs_{n_i}+b)=q^{n_{i+1}-n_i}(2 s_{n_i}+q^{-1}b)\le \frac{s_{n_i}}{4}+\eps.
$$
Since the map $h(x)=x/4+\eps$ has a unique attracting point $4\eps/3$, which attracts all points of $\R$, then
$s_{n_i}$ becomes eventually less than $4\eps$.
Since $\eps>0$ is arbitrary, it follows that $s_n\to 0$ as $i\to\infty$, as desired.
\end{proof}

\subsection{Critical decoration is periodic}
Recall that we consider a fixed immediately renormalizable polynomial $f$.
Consider two $K$-rays $R=R(\al)$ and $L=R(\be)$ landing in $K^*$.
Also, take any equipotential $\wh E$ of $f$.
Let $\Delta=\Delta(R,L,\wh E)$ be the bounded complementary component of $K^*\cup R\cup L\cup \wh E$
such that the external rays that penetrate into $\Delta$ have arguments that belong to the positively oriented arc from $\al$
to $\be$.
Call such sets $\Delta$ \emph{bounded sectors}.
Evidently, $\Delta$ is the intersection of $S^\circ(\al,\be)$ with the Jordan disk enclosed by $\wh E$.
Hence results of the previous section dealing with sectors apply to $\Delta$ and similar sets.
In particular, $\arg(\Delta)$ can be defined as the set of arguments of all $K$-rays intersecting $\Delta$.
For $\Delta$ defined above, $\arg(\Delta)=(\alpha,\beta)$.

Let $\Delta'$ be an iterated pullback of $\Delta$ such that $\ol{\Delta'}\cap K^*\ne\0$; then say that $\Delta'$ is a pullback
of $\Delta$ \emph{adjacent} to $K^*$. If $\Delta'$ is an $f^n$-pullback of $\Delta$ such that $\ol{\Delta'}\cap K^*\ne\0$,
we call $\Delta'$ an \emph{iterated} pullback of $\Delta$ \emph{adjacent} to $K^*$.
Let $\Delta=\Delta_0$ and let, for every $n$, the set $\Delta_n$ be a pullback of $\Delta_{n-1}$ adjacent to $K^*$.
Then the sequence of sets $\Delta_n, n=0, 1, \dots$ is called a \emph{backward pullback orbit of
 bounded sectors adjacent to $K^*$}.
For it we define the set $\Nf=\{n_1<n_2<\dots\}$ of \emph{all} positive integers such that $\om_2\in\Delta_{n_i}$
($\Nf$ may be finite or infinite); in the notation we suppress the dependence on the sets $\{\Delta_n\}$.
By Lemma \ref{l:pb-opsec}, each $\Delta_n$ has $\arg(\Delta_n)=(\al_n,\be_n)$ for some $\al_n$ and $\be_n$.
Set $I_n=[\al_n, \be_n]$.

Set $\wt U^*$ to be the pullback of $U^*$ containing $\wt K^*$.
Recall that we consider polynomials $f\in \imr$ such that $\om_2$ is non-recurrent.

\begin{thm}\label{t:del-bnd}
Fix a bounded sector $\Delta=\Delta_0$ and a backward pullback orbit $\{\Delta_n\}$ of $\Delta$ adjacent to $K^*$.
Then one of the following holds.

\begin{enumerate}

\item The set $\Nf$ is infinite, $n_{i+1}-n_i\not \to\infty$ and there exists a number $N$ such that $n_{i+1}-n_i$ takes the same value
$N$ infinitely many times. Then the quadratic argument of $\om_2$ is $\si_2^N$-fixed, the quadratic argument of some
image of $\om_2$ belongs to $I_0$, the arcs $I_n$ are of length $|I_0|/2^n$, and $I_{n_i}$ contains $\om_2$ for any $i$.

\item The set $\{n_1<n_2<\dots\}$ is finite, and, for a number $M$ and all $m>M$,
 we have $\Delta_m\subset U^*$.

\end{enumerate}

\end{thm}

\begin{proof}
Consider a finite covering $\Uc$ of $\Delta\sm U^*$ by Ma\~ne neighborhoods.
Similarly, fix a finite covering $\Vc$ of $\wt U^*$ by Ma\~ne neighborhoods such that $\bigcup\Vc=\wt U^*$.
Set $\Uc_1=\Uc$ and define $\Uc_n$ inductively as follows.
Assuming by induction that $\Delta_n\sm U^*\subset\bigcup\Uc_n$,
 define $\Uc_{n+1}$ as the set of all open sets $U$ satisfying one of the following:
\begin{enumerate}
  \item[(i)] there is $U'\in\Uc_n$ such that $U\subset\Delta_{n+1}$ is a pullback of $U'$;
  \item[(ii)] the point $\om_2$ is in $\Delta_{n+1}$, and $U\in\Vc$.
\end{enumerate}
Neighborhoods in $\Uc_{n+1}$ as in item (i) (resp., (ii)) are called \emph{type} (i) (resp., \emph{type} (ii)) neighborhoods.
Observe that properties of $\Delta_{n+1}$ as a pullback of $\Delta_n$ adjacent to $K^*$ are described in
Lemmas \ref{l:pb-opsec} and \ref{l:img-sec}.

Any $U\in\Uc_n$ is either obtained as an $f^k$-pullback of some type (i) neighborhood in $\Uc_{n-k}$
with $k$ being maximal with this property, or comes from $\Vc$ but only at the moments when $\om_2\in \Delta_n$.
In the former case set $s(U)=Cq^k$; by Theorem \ref{t:mane},
 we have $\mathrm{diam}(U)\le s(U)$. In the latter case set $s(U)=\mathrm{diam}(U)$.
Define $s_n=s_n(\Delta)$ as the sum of $s(U)$ over all $U\in\Uc_n$.
By the triangle inequality $\mathrm{diam}(\Delta_n\sm U^*)$ is bounded from above by $s_n$.

We claim that if $\Nf$ is finite or if $n_{i+1}-n_i\to \infty$ as $i\to \infty$, then
$s_n\to 0$ as $n\to\infty$. The former immediately follows from properties of neighborhoods of type (i).
This completes case (2) of the theorem.
 To deal with the latter ($n_{i+1}-n_i\to \infty$ implies $s_n\to 0$),
consider two cases of transition from $n$ to $n+1$.

(a) Assume that $\om_2\notin \Delta_{n+1}$.
Then, by Lemma \ref{l:pb-opsec}, the bounded sector
 $\Delta_{n+1}$ maps one-to-one to $\Delta_n$ and no point of $\Delta_{n+1}\sm U^*$ maps into $U^*$.
It follows that all neighborhoods in $\Uc_{n+1}$ are of type (i) and the number of neighborhoods remains the same so that
$s_{n+1}=q s_n$.

(b) Assume that $\om_2\in\Delta_{n+1}$.
Then there are at most twice as many type (i) neighborhoods in $\Uc_{n+1}$ as neighborhoods in $\Uc_n$.
Also, $\Uc_{n+1}$ includes $\Vc$.
We conclude that $s_{n+1}\le 2qs_n+
\mathrm{diam}(\wt U^*)$.

Thus, $s_n$ satisfies Lemma \ref{l:sn} with
$b=\mathrm{diam}(\wt U^*)$, and numbers $n_i$ are exactly the bad subscripts from that lemma.
By Lemma \ref{l:sn}, we have $s_{n_i}\to 0$ as $i\to\infty$.
Replacing $U^*$ with a smaller neighborhood $\widehat U^*$ of $K^*$ and repeating the same argument
for $\widehat U^*$ and the original $\Delta$ yields the existence of $N$ such that $\Delta_N\subset U^*$
from which moment on the pullbacks of $\Delta_N$ adjacent to $K^*$ are contained in $U^*$ and cannot contain
$\om_2$, a contradiction with the assumption. This proves that $n_{i+1}-n_i\not\to \infty$.

Since $n_{i+1}-n_i\not\to \infty$, the desired number $N$ exists.
Since $\Delta_{n+1}$ is an $f$-pullback of $\Delta_n$, then $\al_n=2\al_{n+1}\pmod 1$ and $\be_n=2\be_{n+1}\pmod 1$,
and the interval $I_{n+1}$ is twice shorter than $I_n$.
Evidently, $\om_2\in I_{n_i}$ for any $i$.
By the assumption, $n_{i+1}=n_i+N$ for infinitely many numbers $i$; for these numbers,
$\om_2\in\Delta_{n_i}\cap\Delta_{n_i+N}$, which implies that the quadratic argument of $\om_2$ belongs
to $I_{n_i}=\si_2(I_{n_i+N})\cap I_{n_i+N}$. Passing to the limit, we see that the quadratic argument is $\si_2^N$-fixed.
\end{proof}

\begin{lem}\label{l:all-n}
If a decoration $D$ has the quadratic argument $\arg_2(D)$
which does not belong to the orbit of a periodic quadratic argument $\arg_2(D_c)$,
then there exists $M$ such that any $f^M$-pullback of $D$ adjacent to $K^*$ is contained
in $U^*$. In particular, if $\arg_2(D_c)$ is not periodic, then this holds for any decoration $D$.
\end{lem}

\begin{proof}
By the assumption we can choose a bounded sector $\Delta$
 such that $\arg_2(D)\in \tau(\arg(\Delta))$,
 and $\tau(\arg(\Delta))$ is disjoint from the orbit of a periodic quadratic argument $\arg_2(D_c)$.
By Theorem \ref{t:del-bnd}, the lemma follows.
\end{proof}

All this implies Proposition \ref{p:alldec}.

\begin{prop}
  \label{p:alldec}
Every decoration $D$ is eventually mapped to $D_c$.
\end{prop}

\begin{proof}
Consider the union $K_d$ of $K^*$ and all decorations that are eventually mapped to $D_c$.
We claim that the set $K_d$ is backward invariant. Indeed, take any decoration $D\subset K_d$.
Then any $f$-pullback of $D$ is either a decoration in $K_d$ or a subset of $D_c$.
Thus, $f^{-1}(K_d)\subset K_d$.

Now, suppose that the quadratic argument $\arg_2(D)=\ga$ of $D$ does not belong to
a periodic orbit of $\arg_2(D_c)$.
Then, by Lemma \ref{l:all-n}, for any $U^*$, the $f^n$-pullbacks
 of $D$ adjacent to $K^*$ will be contained in $U^*$ for any $n>M_D(U^*)$,
 where $M_D(U^*)$ depends on $D$ and $U^*$.
Therefore, the set $K_d$ is compact as a union of $K^*$ and a sequence of sets
that are closed in $\C\sm K^*$ and accumulate to $K^*$. Observe that if $\arg_2(D_c)$ is
periodic, the sets in $K_d$ are decorations with periodic arguments from the
$\si_2$-orbit of $\arg_2(D_c)$, or decorations with non-periodic arguments that are
 iterated preimages of $\arg_2(D_c)$.

Clearly, $K_d$ is a full continuum. To sum it all up, $K_d$ is a full subcontinuum of
$K$, which is backward invariant. It follows that $K_d=K$, which completes the proof.
\end{proof}

The next lemma specifies properties of the gap $\Uf$.

\begin{lem}\label{l:perio-type}
If $\om_2$ is non-recurrent, then $\Uf$ is of periodic type.
\end{lem}

\begin{proof}
Consider the critical value decoration $D_v$.
By Proposition \ref{p:alldec}, $D_v$ eventually maps back to $D_c$.
It follows that the quadratic argument $\arg_2(D_c)$ is periodic.
Therefore, $\Uf$ is of periodic type by Lemma \ref{l:cridi3}.
\end{proof}

\subsection{Major hole defines a cut}
\label{ss:majcut}
In Section \ref{ss:majcut}, we complete the proof of Theorem \ref{t:recur}.
We assume that $\om_2$ is not recurrent, which implies, by Lemma \ref{l:perio-type},
 that $\Uf$ has periodic type.
Let $(\al_\Uf,\be_\Uf)$ be the major hole of $\Uf$, and let $k$ be the minimal $\si_3$-period of $\al_\Uf$.

\begin{lem}
  \label{l:al-land}
  Both rays $R(\al_\Uf)$ and $R(\be_\Uf)$ land in $K^*$.
\end{lem}

\begin{proof}
  It is enough to prove the claim for $R(\al_\Uf)$.
Assume the contrary: $R(\al_\Uf)$ lands on a point $a\notin K^*$.
Choose a preperiodic argument $\al$ in $\Uf$ sufficiently close to $\al_\Uf$
that is not an (eventual) preimage of $\al_\Uf$ or $\be_\Uf$.
By Theorem \ref{t:who-land}, the ray $R(\al)$ lands in $K^*$.
Define $W$ as the complementary component of $R(\al)\cup R(\al_\Uf)\cup K$ containing
 all rays with arguments in $(\al,\al_\Uf)$.
There is a unique univalent $f^k$-pullback $W_1$ of $W$ that is contained in $W$.
In fact, $W_1$ is bounded by $R(\al_1)$, $R(\al_\Uf)$, and a part of $K$,
 where $\si_3^k(\al_1)=\al$ and $\al<\al_1<\al_\Uf$.
Let $g:W\to W_1$ be the inverse of $f^k:W_1\to W$.

Note that $a$ attracts all nearby points of $W$ under the iterates of $g$.
That is, $g^n(x)\to a$ for any $x\in W$ sufficiently close to $a$.
Indeed, observe that $W$ coincides with $W_1$ near $a$, and, for this reason,
 the local inverse of $f^k$ near $a$ coincides with $g$.
On the other hand, there are points of $W$ that converge to $K^*$ under the iterates of $g$.
To see that, it suffices to take any point of $R(\al_1)$ that lies in $U^*$
and use the definition of a PL set.

Take two points $x$, $y\in W$ such that $g^n(x)\to K^*$ and $g^n(y)\to a$ as $n\to\infty$.
On the one hand, $g:W\to W$ is a hyperbolic contraction, hence the sequence
 $\mathrm{dist}_W(g^n(x),g^n(y))$ is bounded, where $\mathrm{dist}_W$ means the hyperbolic distance in $W$.
On the other hand, as $x_n$ and $y_n$ converge to different boundary points of $W$,
 the distance $\mathrm{dist}_W(x_n,y_n)$ tends to infinity.
In particular, passing to subsequences, we see that $\mathrm{dist}_W(g^n(x),g^n(y))$ is unbounded,
 a contradiction.
\end{proof}

Theorem \ref{t:sameland} implies Theorem \ref{t:recur}.

\begin{thm}
  \label{t:sameland}
  The rays $R(\al_\Uf)$ and $R(\be_\Uf)$ land on the same point.
\end{thm}

\begin{proof}
By Lemma \ref{l:al-land}, both $R(\al_\Uf)$ and $R(\be_\Uf)$ land in $K^*$.
Assume by way of contradiction that the landing point $a$ of $R(\al_\Uf)$ is different from
 the landing point $b$ of $R(\be_\Uf)$.
Let $W$ be a component of $S^\circ(\al_\Uf,\be_\Uf)\cap U^*$ such that $a$, $b\in\bd(W)$.
Clearly, such $W$ exists, and a univalent $f^k$-pullback $W_1$ of $W$ is a subset of $W$.
The rest of the proof is a hyperbolic distance argument similar to that used in the proof of Lemma \ref{l:al-land}.
Namely, choose points $x$, $y\in W$ so that $g^n(x)\to a$ and $g^n(y)\to b$.
By hyperbolic contraction of $g$, the hyperbolic distance (w.r.t. $W$) between $g^n(x)$ and $g^n(y)$ is bounded.
On the other hand, it must diverge to infinity since $a\ne b$, a contradiction.
\end{proof}

In Theorem \ref{t:recur}, we assume that $f$ has no periodic cuts.
Yet, as we have just proved (see Theorem \ref{t:sameland}), if $\om_2$ is non-recurrent, then $f$ must have a periodic cut.
This contradiction proves Theorem \ref{t:recur}.

\end{document}